\documentclass{article}

\usepackage[english]{babel}
\usepackage[utf8x]{inputenc}
\usepackage[T1]{fontenc}
\usepackage{amsmath,amssymb,amsthm,mathtools,mathdots,nicefrac,tikz}
\usepackage{authblk}

\usepackage[a4paper,top=3cm,bottom=2cm,left=3cm,right=3cm,marginparwidth=1.75cm]{geometry}

\usepackage{comment}

\usepackage{amsmath}
\usepackage{amsthm}
\usepackage{amssymb}
\usepackage{graphicx}
\usepackage{caption}
\usepackage[colorinlistoftodos]{todonotes}
\usepackage[colorlinks=true, allcolors=blue]{hyperref}
\usepackage{enumerate}
\usepackage{algorithm}
\usepackage[noend]{algpseudocode}

\makeatletter
\def\BState{\State\hskip-\ALG@thistlm}
\makeatother

\usetikzlibrary{math,graphs,cd}
\usepackage{xcolor}

\usepackage{caption}

\theoremstyle{definition} % non-italic theorems
\newtheorem{example}{Example}[section]

\theoremstyle{plain} % non-italic theorems
\newtheorem{theorem}{Theorem}[section]
\newtheorem{corollary}{Corollary}[section]
\newtheorem{lemma}{Lemma}[section]
\newtheorem{proposition}{Proposition}[section]

\newcommand{\cS}{\mathcal{S}}

\newcommand{\sn}{\mathrm{sn}}

\DeclareMathOperator{\outdeg}{outdeg}

\DeclareMathOperator{\val}{val}

\DeclareMathOperator{\gon}{gon}

\usepackage[colorinlistoftodos]{todonotes}

\title{Uniform scrambles on graphs}
\author[1]{Lisa Cenek}
\author[2]{Lizzie Ferguson}
\author[3]{Eyobel Gebre}
\author[4]{Cassandra Marcussen}
\author[5]{Jason Meintjes}
\author[2]{Ralph Morrison}
\author[2]{Liz Ostermeyer}
\author[6]{Shefali Ramakrishna}
\affil[1]{University of Illinois Chicago, Chicago, IL, USA}
\affil[2]{Williams College, Williamstown, MA, USA}
\affil[3]{University of Pennsylvania, Philadelphia, PA, USA}
\affil[4]{Harvard University, Cambridge, MA, USA}
\affil[5]{San Francisco State Universityt, San Francisco, CA, USA}
\affil[6]{Cornell University, Ithaca, NY, USA}
\date{}                     %% if you don't need date to appear
\usepackage{graphicx}

\begin{document}

\maketitle

\begin{abstract}
A scramble on a connected multigraph is a collection of connected subgraphs that generalizes the notion of a bramble.  The maximum order of a scramble, called the scramble number of a graph, was recently developed as a tool for lower bounding divisorial gonality.  We present results on the scramble of all connected subgraphs with a fixed number of vertices, using these to calculate scramble number and gonality both for large families of graphs, and for specific examples like the \(4\)- and \(5\)-dimensional hypercube graphs.  We also study the computational complexity of the egg-cut number of a scramble.
\end{abstract}

\section{Introduction}
A \emph{scramble} on a graph \(G\) is a collection \(\mathcal{S}=\{V_1,\ldots,V_\ell\}\) of connected subsets \(V_i\subset V(G)\); we call each \(V_i\) an \emph{egg} of the scramble.  The \emph{order} of \(\mathcal{S}\) is the minimum of two numbers:  \(h(\mathcal{S})\), the minimum size of a hitting set of vertices for \(\mathcal{S}\); and \(e(\mathcal{S})\), the minimum number of edges required to be deleted to put two eggs into separate components.  The \emph{scramble number} \(\sn(G)\) of a graph is the maximum possible order of a scramble.

Scrambles were introduced in \cite{scramble} as a generalization of brambles\footnote{A bramble is also a collection of connected subsets of vertices of a graph, with the added condition that the union of any two of those subsets is also connected.  The order of a bramble is simply the minimum size of a hitting set.} for the purpose of studying \emph{(divisorial) gonality}.  The gonality \(\gon(G)\) of a graph measures how difficult it is to win a certain ``chip-firing'' game on the graph, with motivation coming from algebraic geometry \cite{baker}.  The key property of scramble number that helps in the study of gonality is its capacity to provide the lower bound \(\sn(G)\leq\gon(G)\).  This has been used to compute the gonality of many graphs, including stacked prism graphs and toroidal grid graphs \cite{scramble}; many other Cartesian product graphs \cite[\S 5]{echavarria2021scramble}; and any simple graph with sufficiently high minimum valence \cite[\S 3]{echavarria2021scramble}.

In this paper we introduce the
\emph{\(k\)-uniform scramble} \(\mathcal{E}_k\) on a graph \(G\), whose eggs are the connected subsets of \(V(G)\) consisting of exactly \(k\) vertices.  We prove that the order of this scramble is determined by two previously studied graph invariants:  the \(k\)-restricted edge-connectivity \(\lambda_k(G)\) (also known as order \(k\) edge-connectivity), and the \(\ell\)-component independence number \(\alpha_\ell^c(G)\) (these invariants are defined in Section \ref{section:parameters}).  In particular, we prove in Theorem \ref{theorem:order_ek} that the order of \(\mathcal{E}_k\) is
\[||\mathcal{E}_k||=\min\{\lambda_k(G),n-\alpha_{k-1}^c(G)\}\]
where \(n\) denotes the number of vertices of \(G\).  We use this to prove our main theorem.

\begin{theorem}\label{theorem:main}
Fix \(\ell\geq 3\), and let \(G\) be a graph with girth at least \(\ell\).  Then \(\gon(G)\leq n-\alpha_{\ell-2}^c(G)\). Moreover, if \(\lambda_{\ell-1}(G)\geq n-\alpha_{\ell-2}^c(G)\), then
\[\sn(G)=\gon(G)=n-\alpha_{\ell-2}^c(G).\]
\end{theorem}
We use this theorem to determine gonality for graphs of increasing girth and decreasing minimum valence (Theorems \ref{theorem:girth3} through \ref{theorem:girth5}), as well as for bipartite graphs with high minimum valence (Theorems \ref{theorem:bipartite1} and \ref{theorem:bipartite2}) and for the hypercube graphs \(Q_4\) and \(Q_5\) (Proposition \ref{theorem:gonalityQ4} and Theorem \ref{theorem:gonalityQ5}).

Our paper is organized as follows.  In Section \ref{section:parameters} we present necessary background on graph parameters and some useful lemmas.  In Section \ref{section:k_uniform} we study the \(k\)-uniform scramble, and prove Theorem \ref{theorem:main} as well as many subsequent results.  In Section \ref{section:np_hard} we study questions related to the computational complexity of the number \(e(\mathcal{S})\).

\medskip

\section{Graph parameters}
\label{section:parameters}

We let \(G=(V,E)\) be an undirected multigraph with vertex set \(V(G)\) and edge multiset \(E(G)\), where edges are not allowed from a vertex to itself.  Throughout this paper  we assume our graph \(G\) is connected, so that there exists a path from any vertex to any other vertex.  The \emph{valence} \(\val(v)\) of a vertex \(v\) is the number of edges incident to it.  We let \(\delta(G)\) denote the minimum valence of a vertex in \(V(G)\).  Given two subsets (possibly overlapping) \(A,B\subset V(G)\), we let \(E(A,B)\subset E(G)\) be the multiset of edges that have one endpoint in \(A\) and the other endpoint in \(B\). The \emph{outdegree} of a set \(S\subset V(G)\), denoted \(\textrm{outdeg}(S)\), is \(|E(S,S^C)|\), where \(S^C\) denotes the complement of \(S\) in \(V(G)\). Given \(S\subset V(G)\), we let \(G[S]\) denote the subgraph with vertex set \(S\) and edge set \(E(S,S)\).  If \(G[S]\) is a connected subgraph, we call \(S\) \emph{connected}.

The \emph{girth} of a graph is the length of a smallest cycle in the graph, taken to be infinity if the graph has no cycles.  A graph has girth at least \(3\) if and only if it is a simple graph; and a simple graph has girth at least \(4\) if and only if no three vertices form a triangle \(K_3\).  For this reason simple graphs of girth at least \(4\) are also called \emph{triangle-free} graphs. 

Given a multiset \(T\subset E(G)\), we let \(G-T\) denote the subgraph of \(G\) with vertex set \(V(G)\) and edge multiset \(E(G)\setminus T\). For \(G\) a graph on more than \(1\) vertex, the \emph{edge-connectivity} \(\lambda(G)\) is the minimum cardinality of a set \(T\subset E(G)\) with \(G-T\) disconnected.  More generally, for \(k\geq 1\), the \emph{\(k\)-restricted edge-connectivity} of a graph \(G\), denoted \(\lambda_k(G)\), is the minimum size of a set $T \subset E(G)$ (if it exists)  such that $G - T$ is disconnected and each connected component of $G -T$ contains at least $k$ vertices. If no such set \(T\) exists, we set \(\lambda_k(G)=\infty\) throughout this paper\footnote{This is not necessarily the standard convention; for instance, it implies that \(\lambda_1(K_1)=\infty\) rather than the standard convention of \(\lambda_1(K_1)=0\).}. When \(\lambda_k(G)\) is finite, we say that \(G\) is \(\lambda_k\)-connected.

For a graph \(G\) we let \(\xi_k(G)\) denote the minimum outdegree of any connected subgraph of \(G\) on \(k\)-vertices. 
Letting \(G\) be a graph with $\lambda_k(G)\leq \xi_k(G)$, we say that $G$ is \emph{$\lambda_k$-optimal} if $\lambda_k(G) = \xi_k(G)$.  Since \(\xi_k(G)\) is in practice often easier to compute than \(\lambda_k(G)\), it is useful to have classes of graphs that are known to be \(\lambda_k\)-optimal.  We recall the following results in this vein.

\begin{lemma}[\cite{wang_li}]\label{lemma:wang_li}
Let \(G\) be a \(\lambda_2\)-connected graph.  If \(\val(u)+\val(v)\geq |V(G)|+1\) for all pairs \(u,v\) of non-adjacent vertices, then \(G\) is \(\lambda_2\)-optimal.
\end{lemma}

%\begin{lemma}\label{theorem: conditions for optimal} (Theorem 6 in \cite{xuxu})
%If $G$ is a connected, vertex-transitive graph such that either $G$ does not contain $K_3$ as a subgraph, or $|V(G)|$ is odd, then $G$ is \(\lambda_2\)-optimal. 
%\end{lemma}

\begin{lemma}[Corollary 3.2 in \cite{mv}]\label{lemma:mv}
Every connected triangle-free graph of order \(n\) with \(\delta(G)\geq 3\) and \(\xi_3(G)\geq n+1\) is \(\lambda_3\)-optimal.
\end{lemma}

\begin{lemma}[Corollary 4.9 in \cite{hmm}]\label{lemma:hmm}
Let \(G\) be a connected and triangle-free graph of order \(n\geq 2k\).  If \(\delta(G)\geq \frac{1}{2}\left(\lfloor n/2\rfloor +k\right)\), then \(G\) is \(\lambda_k\)-optimal.
\end{lemma}

%\begin{lemma}[Theorem 3.4 in \cite{hmm}]
%Let \(G\) be a connected triangle-free graph on \(n\geq 6\) vertices.  If \(\val(u)+\val(v)\geq 2\lfloor n/4\rfloor +3\) for each pair \(u,v\) of non-adjacent vertices, then \(G\) is \(\lambda_3\)-optimal.
%\end{lemma}

%\begin{lemma}[Corollary 3.1 in \cite{hv}]\label{lemma:hv}
%Let \(G\) be a \(\lambda_2\)-connected triangle-free graph.  If \(\val(v)\geq \lfloor (n+2)/4\rfloor +1\) for all vertices \(v\) with at most one exception, then \(G\) is \(\lambda_2\)-optimal.
%\end{lemma}

A set \(S\subset V(G)\) is called an \emph{independent set} if \(G[S]\) consists of isolated vertices, or equivalently if \(E(S,S)=\emptyset\).  The \emph{independence number} \(\alpha(G)\) of a graph \(G\) is the maximum cardinality of an independent set.  There are a number of ways of generalizing independence number, one of which we recall here.  For \(\ell\geq 0\), an \emph{\(\ell\)-component independent set} \(S\subset V(G)\) is one such that \(G[S]\) has all components of order at most \(\ell\).  Thus a \(1\)-component independent set is a usual independent set, and the only \(0\)-component independent set is the empty set.  The \emph{\(\ell\)-component independence number} of a graph \(G\), denoted \(\alpha^c_\ell(G)\), is then the maximum size of an \(\ell\)-component independent set.  Thus \(\alpha_0^c(G)=0\);  \(\alpha_1^c(G)=\alpha(G)\), the usual independence number of \(G\); and \(\alpha_2^c(G)=\textrm{diss}(G)\), the \emph{dissociation number} of \(G\), which is the largest collection of vertices to induce a subgraph consisting of a disjoint union of vertices and \(K_2\)'s.

We now recall our scramble definition\footnote{Our definition is phrased slightly differently from the original definition in \cite{scramble}, but is equivalent by \cite[\S 2]{echavarria2021scramble}.} and notation. A \emph{scramble}  on a graph \(G\) is a collection \(\mathcal{S}=\{V_1,\ldots,V_\ell\}\) of connected subsets \(V_i\subset V(G)\), where each $V_i$ is called an \emph{egg}.  A \emph{hitting set} of \(\mathcal{S}\) is a subset \(C\subset V(G)\) such that for all \(i\) we have \(C\cap V_i\neq\emptyset\).  We let \(h(\mathcal{S})\) denote the minimum size of a hitting set.  An \emph{egg-cut} of \(\mathcal{S}\) is a subset \(A\subset V(G)\) such that \(A\) and \(A^C\) each contain an egg; the \emph{size} of an egg-cut is \(|E(A,A^C)|\).  We let \(e(\mathcal{S})\) denote the minimum size of an egg-cut.  The \emph{order} of a scramble is then defined to be
\[||\mathcal{S}||=\min\{h(\mathcal{S}), e(\mathcal{S})\}.\]
The \emph{scramble number} \(\sn(G)\) of a graph is the maximum possible order of a scramble.

\begin{lemma}\label{lemma:min_with_connected}
Let \(\mathcal{S}\) be a scramble on a connected graph \(G\).  The number \(e(\mathcal{S})\) is unchanged if we restrict to egg-cuts \((A,A^C)\) such that \(A\) and \(A^C\) are both connected.
\end{lemma}

\begin{proof}
Given an arbitrary egg-cut \((A,A^C)\), we will show that there exists another egg-cut \((A'',A''^C)\) with \(A''\), \(A''^C\) both connected and \(|E(A'',A''^C)|\leq |E(A,A^C)|\); it will follow that the minimum size of an egg-cut is achieved with both sets connected.

Let \((A,A^C)\) be an egg-cut, with \(V_1\subset A\) and \(V_2\subset A^C\) for some distinct eggs \(V_1,V_2\in\mathcal{S}\).  Since \(V_1\) is connected, there exists a connected component \(H_1\) of \(G[A]\) containing \(V_1\).  Letting \(A'=V(H_1)\), we have that \((A',A'^C)\) is an egg-cut, since \(V_1\subset A'\) and \(V_2\subset A'^C\). We also have \(|E(A',A'^C)|\leq |E(A,A^C)|\):  any edge from \(A'\) to \(A'^C\) must have been an edge from \(A'\) to \(A^C\), since \(A'\) was a component of \(G[A]\) and thus had no edges to any vertex in \(A-A'\).

If \(A'^C\) is connected, we are done.  If not, let \(H_2\) be the connected component of \(G[A'^C]\) containing \(V_2\), and let \(A''=V(H_2)^C\).  Then \((A'',A''^C)\) is an egg-cut, with \(V_1\subset A''\) and \(V_2\subset A''^C\).  We have that \(A''^C=V(H_2)\) is connected. In fact, so is \(A''\), as it is the union of \(A'\) with all vertices from the connected components of \(G[A'^C]\) besides \(H_2\); since $G$  is connected, \(A'\) shares edges with each such component.  Finally, by the same argument as before we have \(|E(A',A'^C)|\leq |E(A'',A''^C)|\leq |E(A,A^C)|\), as desired.

\end{proof}

\begin{example}\label{example:herschel_sn}
Let \(G\) be the Herschel graph, depicted twice in Figure \ref{figure:herschel}; this is the smallest polyhedral graph that does not have a Hamiltonian cycle.  
Let \(\mathcal{S}\) be the scramble on \(G\) consisting of all \(V_i\) such that \(G[V_i]\) is a path on three vertices.  We claim that \(h(\mathcal{S})=5\).  Certainly there exists a hitting set of cardinality \(5\), namely the smaller of the two partite sets.  

\begin{figure}[hbt]
    \centering
    \includegraphics{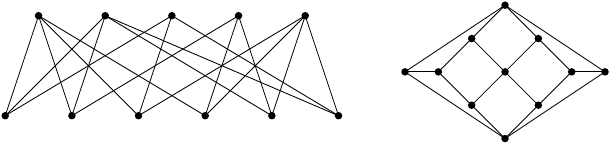}
    \caption{Two drawings of the Herschel graph.}
    \label{figure:herschel}
\end{figure}

Suppose for the sake of contradiction that there exists a hitting set \(C\) with \(|C|\leq 4\).  To reach a contradiction, it will suffice to show that some uncovered vertex has at least two uncovered neighbors.  Let \(A\) be the partite set with more elements from \(C\), and \(B\) be the other partite set.
If \(|A\cap C|=4\), then there exists at least one uncovered vertex in $A$ with at least three uncovered neighbors. If \(|A\cap C|=3\), there exist at least 2 uncovered vertices in \(A\), call them \(u\) and \(v\), with valence at least 3, so no single vertex in the other partite class $B$ will cover more than one neighbor of $u$ and one neighbor of $v$. This means both $u$ and $v$ have at least two unhit neighbors.  Finally, consider the case \(|A\cap C|=2\).  We remark that for any choice of \(w,z\in A\), we have $N(A\setminus\{w,z\})=B$. Because all of the unhit vertices in $A$ have valence at least 3, no two vertices in $B$ from \(C\) will alone be enough to reduce the number of unhit neighbors for each of the vertices in $A\setminus\{w,z\}$ below two. Hence, there is no hitting set with only four vertices, so $h(\cS)=5$.

Now we bound \(e(\cS)\).  By Lemma \ref{lemma:min_with_connected}, it suffices to consider egg-cuts \(A\) such that \(A\) and \(A^C\) are both connected, and without loss of generality we may assume \(|A|\leq |A^C|\); thus \(e(\mathcal{S})=\min\{\xi_3(G),\xi_4(G),\xi_5(G)\}\), since \(3\leq |A|\leq 5\). For \(\xi_3(G)\), we note that every connected subgraph on \(3\) vertices is a path \(P_3\), which will have outdegree at least \(3\delta(G)-4=9-4=5\), since each vertex has degree at least \(3\) but \(2\cdot 2\) edges are double-counted inside of \(P_3\); thus \(\xi_3(G)=5\).  For \(\xi_4(G)\), a connected subgraph on \(4\) vertices is either a tree with outdegree at least \(4\delta(G)-2\cdot 3=12-6=6\); or a cycle  \(C_4\).  Such a cycle a priori has outdegree at least \(4\delta(G)-2\cdot 4=4\); but such a small outdegree would only be possible if all vertices had valence \(3\), which does not occur in any \(C_4\) subgraph of \(G\), giving outdegree at least \(5\).  Finally, for \(\xi_5\), a connected subgraph on \(5\) vertices is either a tree with outdegree at least \(5\delta(G)-2\cdot 4=15-8=7\); or a cycle  \(C_4\) with an additional edge and leaf.  Compared to \(C_4\), the added edge removes \(1\) from the outdegree and adds \(2\), giving outdegree at least \(1+5=6\).  We conclude that \(e(\mathcal{S})=\min\{\xi_3(G),\xi_4(G),\xi_5(G)\}=5\).

This gives us that \(||\mathcal{S}||=\min\{5,5\}=5\), so \(\sn(G)\geq 5\).  We will later see that in fact \(\sn(G)=5\).
\end{example}

We now briefly recall divisor theory on graphs.  A \emph{divisor} \(D\) on a graph \(G\) is a \(\mathbb{Z}\)-linear combination of the vertices of \(G\):
\[D=\sum_{v\in V(G)}a_v\cdot (v),\,\,\,\,\,\,\,a_v\in\mathbb{Z}.\]
We denote the coefficient of \((v)\) by \(D(v)\); that is, \(D(v)=a_v\).
The set of all divisors \(\textrm{Div}(G)\) on a graph \(G\) forms a group, namely the free abelian group on the set \(V(G)\).  The \emph{degree} of a divisor is the integer sum of the coefficients:
\[\deg(D)=\sum_{v\in V(G)}a_v.\]We interpret a divisor on a graph \(G\) as a placement of integer numbers of poker chips on the vertices of the graph \(G\), with negative integers representing debt.  Under this interpretation the degree of a divisor is the total number of chips.

The \emph{chip-firing move at \(v\)} transforms a divisor \(D\) into a divisor \(D'\) by removing \(\val(v)\) chips from \(v\) and giving \(|E({v},{w})|\)  chips to every vertex \(w\neq v\):
\[D'=(a_v-\val(v))(v)+\sum_{w\in V(G)-\{v\}}(a_w+|E({v},{w})|)(w).\]
That is, \(v\) donates chips to its neighbors, one along each edge incident to \(v\).  We say two divisors \(D\) and \(E\) are \emph{equivalent}, written \(D\sim E\), if we can transform \(D\) into \(E\) via  a sequence of chip-firing moves.  Figure \ref{figure:chip_firing_example} illustrates a collection of divisors that are equivalent to one another, with each divisor obtained from the last by firing the circled vertex.  We remark that it does not matter what order vertices are fired, only which vertices are fired and how many times: at each vertex, the change in the number of chips is determined only by how many times it fires, and how many times its neighbors fire, keeping track of parallel edges. This allows us to consider \emph{subset firing moves}, in which all vertices in a subset \(S\) are ``simultaneously'' chip-fired.  The net effect of this is for a chip to be moved along every edge of \(E(S,S^C)\) from its vertex in \(S\) to its vertex in \(S^C\).  For instance, the rightmost divisor in Figure \ref{figure:chip_firing_example} can be obtained from the leftmost by simultaneously firing all the vertices except the right vertex. 

\begin{figure}[hbt]
    \centering
    \includegraphics{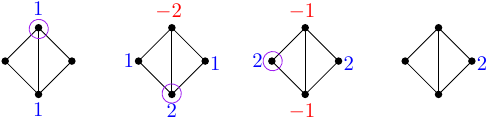}
    \caption{A collection of equivalent divisors; moving from left to right, the next divisor is obtained by chip-firing the circled vertex. Vertices with no integer labels have zero chips.}
    \label{figure:chip_firing_example}
\end{figure}

A divisor \(D\) is called \emph{effective} if \(D(v)\geq 0\) for all \(v\).  The \emph{rank} \(r(D)\) of a divisor \(D\) is defined as follows.  If \(D\) is not equivalent to any effective divisor, we set \(r(D)=-1\).  Otherwise, we let \(r(D)\) equal the maximum integer \(r\geq 0\) such that for any effective divisor \(E\) of degree \(r\), the divisor \(D-E\) is equivalent to some effective divisor.

The \emph{(divisorial) gonality} of a graph \(G\) is the minimum degree of a positive rank divisor.  Thus, if the gonality of a graph is \(k\), then there exists a divisor \(D\) of degree \(k\) such that, for any vertex \(q\in V(G)\), we can perform chip-firing moves to remove all debt; and there exists no such divisor of degree \(k-1\).

\begin{example}\label{example:herschel_gon}
Let \(G\) again be the Herschel graph, and consider the divisor \(D\) of degree \(5\) that places one chip on each of the five vertices of the smaller partite set \(A\), pictured on the left in Figure \ref{figure:herschel_chips}.  We claim that this divisor has positive rank.  To see this, consider any vertex \(q\in B\), the larger partite set.  Chip-firing the subset \(\{q\}^C\subset V(G)\) moves one chip from every neighbor of \(q\) to \(q\), and makes no other changes. This means that some vertices of \(A\) lose a chip to \(q\), going down to \(0\) chips; that some remain at \(1\) chip; that all vertices of \(B\) remain at \(0\) chips; that that \(q\) gains three chips, since it has three neighbors  in \(A\).  An example of this chip-firing process is illustrated in the middle and right of Figure \ref{figure:herschel_chips}. This strategy works to eliminate debt in any divisor of the form \(D-(q)\) where \(q\in B\); and any divisor of the form \(D-(q)\) where \(q\in A\) is already effective.  Thus \(r(D)\geq 1\), meaning \(\gon(G)\leq \deg(D)=5\).
\end{example}

Providing a lower bound of \(k\) on gonality often proves difficult, since ruling out the existence of \emph{any} divisor of degree \(k-1\) of positive rank is more involved than exhibiting a single divisor of degree \(k\) of positive rank.  Fortunately, the scramble number of a graph serves as a lower bound on gonality.

\begin{theorem}[Theorem 1.1 in \cite{scramble}]\label{theorem:sn_leq_gon}
For any graph \(G\), we have \(\sn(G)\leq \gon(G)\).
\end{theorem}

 Combining Examples \ref{example:herschel_sn} and \ref{example:herschel_gon} with this theorem, we have
\[5\leq \sn(G)\leq \gon(G)\leq 5,\]
meaning that \(\sn(G)=\gon(G)=5\) for the Herschel graph.

\begin{figure}[hbt]
    \centering
    \includegraphics[scale=0.7]{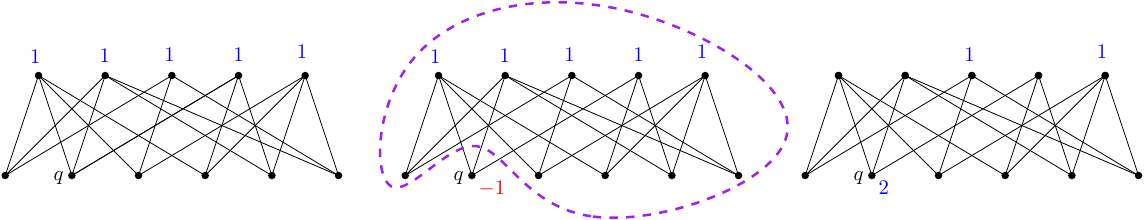}
    \caption{A divisor of degree \(5\), and a chip-firing move to eliminate debt on a vertex \(q\).}
    \label{figure:herschel_chips}
\end{figure}

A useful upper bound on gonality comes in the form of {strong separators}.  A \emph{strong separator} on a graph \(G\) is a non-empty vertex set \(S\subset V(G)\) such that each component \(C\) in \(G-S\) is a tree and every \(v\in S\) is incident to at most one edge with the other endpoint in \(C\).

\begin{lemma}[Theorem 3.33 in \cite{db}]\label{lemma:strong_separator}  If \(S\subset V(G)\) is a strong separator, then \(\gon(G)\leq |S|\).
\end{lemma}

A special case is when \(G\) is a simple graph and \(S\) is the complement of an independent set.  It follows that \(\gon(G)\leq n-\alpha(G)\) for simple graphs, a result also proved in \cite[Proposition 3.1]{gonality_of_random_graphs}.

\section{The \(k\)-uniform scramble}
\label{section:k_uniform}

For a graph \(G\) and an integer \(k\geq 1\), let \(\mathcal{E}_k\) denote the scramble on \(G\) whose eggs are precisely the connected subgraphs on \(k\) vertices.  For \(k=1\), this scramble has each vertex as an egg; this scramble has order \(\min\{|V(G)|,\lambda(G)\}\).  For \(k=2\), the eggs are precisely the sets of the form \(\{u,v\}\) where \(G\) has an edge connecting \(u\) and \(v\); this is the \emph{edge-scramble} defined in \cite[\S 3]{echavarria2021scramble}.

\begin{lemma}
We have \(e(\mathcal{E}_k)=\lambda_k(G)\).
\end{lemma}
\begin{proof}
We will show there exists an egg-cut of size at most \(\ell\) if and only if there exists a set \(T\subset E(G)\) of size at most \(\ell\) with \(G-T\) disconnected and each component of \(G-T\) containing at least $k$ vertices.  First, if there exists an egg-cut \((A,A^C)\) of size at most \(\ell\), there exists another egg-cut \((A'',A''^C)\) of size at most \(\ell\) with \(A''\) and \(A''^C\) both connected by Lemma \ref{lemma:min_with_connected}.  Then choosing \(T=E(A'',A''^C)\) (with size at most \(\ell\)) implies that \(G-T\) is disconnected with two connected components, \(A''\) and \(A''^C\); and since each has an egg, each component has at least \(k\) vertices, meaning the claimed \(T\) exists.  Conversely, if such a \(T\) exists with size at most \(\ell\), then let \(A=V(H)\) for a connected component \(H\) of \(G-T\).  Since \(|A|\geq k\), it contains an egg of \(\mathcal{E}_k\), as odes \(A^C\). Thus \(T\) is an egg-cut of size at most \(\ell\).
\end{proof}

\begin{lemma}
We have \(h(\mathcal{E}_k)=n-\alpha_{k-1}^c(G)\).
\end{lemma}

\begin{proof}
Let \(S\) be a minimum hitting set for \(\mathcal{E}_k\).  Since \(S\) must hit every connected subgraph on \(k\) vertices, we know that \(S^C\) must be a \((k-1)\)-component independent set (any larger component would contain an unhit egg).  This means \(|S^C|\leq \alpha_{k-1}^c(G)\), and so \(h(\mathcal{E}_k) =|S|\geq n-\alpha_{k-1}^c(G)\).  On the other hand, there does exist a hitting set of size \(n-\alpha_{k-1}^c(G)\), namely the complement of any \((k-1)\)-component independent set.  Thus \(h(\mathcal{E}_k)=n-\alpha_{k-1}^c(G)\).
\end{proof}

These lemmas give us the following theorem.

\begin{theorem}\label{theorem:order_ek}
The order of \(\mathcal{E}_k\) is \(\min\{\lambda_k(G),n-\alpha_{k-1}^c(G)\}.\)
\end{theorem}

When \(k=1\), we get the vertex scramble:
\[||\mathcal{E}_1||=\min\{\lambda_1(G),n-\alpha_0(G)\}=\min\{\lambda(G),|V(G)|\}.\]
When \(k=2\), we get the edge scramble:
\[||\mathcal{E}_2||=\min\{\lambda_2(G),n-\alpha_1^c(G)\}=\min\{\lambda_2(G),n-\alpha(G)\}.\]
When \(k=3\), we have
\[||\mathcal{E}_3||=\min\{\lambda_3(G),n-\alpha_2^c(G)\}=\min\{\lambda_3(G),n-\textrm{diss}(G)\}.\]

\begin{proof}[Proof of Theorem \ref{theorem:main}]
Let \(S\) be an \((\ell-2)\)-component independent set of maximum size on \(G\).  We claim that \(S^C\) is a strong separator on \(G\).  To see this, note that \(G-S^C=G[S]\) consists of components with at most \(\ell-2\) vertices, which since the girth of \(G\) is at least \(\ell\) must all be trees.  And if a vertex \(v\in S^C\) were incident to two vertices \(u\) and \(w\) in a component \(C\) of \(G-S^C\), then we have a cycle \(v-u-\ldots-w-v\), where the middle portion of the cycle is the unique path from \(u\) to \(w\) in \(C\); this cycle has at most \(|C\cup\{v\}|\leq \ell-1\) vertices in it, contradicting the girth being at least \(\ell\).  Thus every vertex \(v\in S^C\) is incident to at most one vertex in each component \(C\) of \(G-S^C\), and we have that \(S^C\) is a strong separator.  It follows from Lemma \ref{lemma:strong_separator} that \(\gon(G)\leq |S^C|=n-|S|=n-\alpha_{\ell-2}^c(G)\).

Now assume further that we have \(\lambda_{\ell-1}(G)\geq n-\alpha_{\ell-2}^c(G)\).  The order of the scramble \(\mathcal{E}_{\ell-1}\) lower bounds \(\sn(G)\), which in turn lower bounds \(\gon(G)\) by Theorem \ref{theorem:sn_leq_gon}:
\[||\mathcal{E}_{\ell-1}||=\min\{n-\alpha^c_{\ell-2}(G),\lambda_{\ell-1}(G)\}=n-\alpha^c_{\ell-2}(G)\leq \sn(G)\leq \gon(G)\leq n-\alpha^c_{\ell-2}(G).\]
Since \(\sn(G)\) and \(\gon(G)\) are both upper and lower bounded by \(n-\alpha^c_{\ell-2}(G)\), we have that all three are equal.
\end{proof}

Our first application of this result is to improve on \cite[Corollary 3.2]{echavarria2021scramble}, which states that when \(G\) is a simple graph on \(n\) vertices with \(\delta(G)\geq \lfloor n/2\rfloor +1\), we have \(\sn(G)=\gon(G)=n-\alpha(G)\).

\begin{theorem}\label{theorem:girth3}
Let \(G\) be a simple graph on \(n\) vertices such that for any two adjacent vertices \(u\) and \(v\) we have  \(\val(u)+\val(v)\geq n\), and for any two non-adjacent vertices \(u\) and \(v\) we have \(\val(u)+\val(v)\geq n+1\).  Then we have \(\sn(G)=\gon(G)=n-\alpha(G)\).
\end{theorem}

Note that the assumption \(\delta(G)\geq \lfloor n/2\rfloor +1\) is very slightly more restrictive than our valence assumptions (which allow for one vertex of valence less than \(\lfloor n/2\rfloor +1\)), meaning our result can be viewed as a generalization.

\begin{proof}
First note that if \(G=K_n\), then \(\sn(G)=\gon(G)=n-1=n-\alpha(G)\), and our result holds.  Assume for the rest of the proof that \(G\) is not complete, so \(\alpha(G)\geq 2\).

Since \(G\) is simple, its girth is at least \(3\).  By Theorem \ref{theorem:main} with \(\ell=3\), to show our desired result it suffices to show that \(\lambda_2(G)\geq n-\alpha_1^c(G)=n-\alpha(G)\). Since \(\val(u)+\val(v)\geq n+1\) for non-adjacent \(u\) and \(v\), we know by Lemma \ref{lemma:wang_li} that \(G\) is \(\lambda_2\)-optimal, so we have \(\lambda_2(G)=\xi_2(G)\).  If \(S=\{u,v\}\) with \(\overline{uv}\in E(G)\), we have \(\outdeg(S)=\val(u)+\val(v)-2\geq n-2\geq n-\alpha(G)\) since \(\alpha(G)\geq 2\).   
\end{proof}

Our next two results consider graphs with girth at least \(4\).  The first almost entirely encompasses the second, except the second covers several cases when \(\xi_3(G)<n+1\), such as when three vertices in a path each have valence exactly \(n/3+1\) (giving the path an outdegree of \(n+3-4=n-1\)).

\begin{corollary}\label{theorem:girth4_1}
Let \(G\) be a triangle-free graph with \(\delta(G)\geq3\) and \(\xi_3(G)\geq n+1\).  Then we have
\[\sn(G)=\gon(G)=n-\alpha_2^c(G).\]
\end{corollary}

\begin{proof}
To apply Theorem \ref{theorem:main} with \(\ell=4\), it suffices it suffices to show that \(\lambda_3(G)\geq n-\alpha_2^c(G) \).  By Lemma \ref{lemma:mv}, we know that \(G\) is \(\lambda_3\)-optimal, so \(\lambda_3(G)=\xi_3(G)\geq n+1>n-\alpha_2^c(G)\), as desired.
\end{proof}

\begin{corollary}\label{theorem:girth4_2}
Let \(G\) be a triangle-free graph on \(6\) or more vertices with \(\delta(G)\geq n/3+1\).  Then we have
\[\sn(G)=\gon(G)=n-\alpha_2^c(G).\]
\end{corollary}

\begin{proof}
To apply Theorem \ref{theorem:main} with \(\ell=4\), it suffices to show that \(\lambda_3(G)\geq n-\alpha_2^c(G) \).  Since \(\delta(G)\geq n/3+1\geq \frac{1}{2}(\lfloor n/2\rfloor +3)\), we may apply Lemma \ref{lemma:hmm} with \(k=3\) to deduce that \(G\) is \(\lambda_3\)-optimal, so \(\lambda_3(G)=\xi_3(G)\).  If \(S\) is a connected subgraph on \(3\) vertices, it must be a path on \(3\) vertices since \(G\) is triangle-free, so it has two edges interior to it.  This means that the outdegree of \(S\) is at least \(3\delta(G)-2\cdot 2\geq n+3-4=n-1\), and we have \(\lambda_3(G)\geq n-1\geq n-\alpha_2^c(G)\), as desired.  
\end{proof}

Our last result in this vein is on graphs of girth at least \(5\).

\begin{theorem}\label{theorem:girth5}
Let \(G\) be a graph of girth at least \(5\) on \(8\) or more vertices with \(\delta(G)\geq \frac{1}{2}\left(\lfloor n/2\rfloor+4\right)\).  Then we have
\[\sn(G)=\gon(G)=n-\alpha_3^c(G).\]
\end{theorem}

\begin{proof}
To apply Theorem \ref{theorem:main} with \(\ell=5\), it suffices to show that \(\lambda_4(G)\geq n-\alpha_3^c(G)\).  By Lemma \ref{lemma:hmm} with \(k=4\), we know that \(G\) is \(\lambda_4\)-optimal, so \(\lambda_4(G)=\xi_4(G)\).  Let \(H\) be a connected subgraph of \(G\) on \(4\) vertices.  Since the girth of \(G\) is at least \(5\), \(G\) must be a tree, and so has \(3\) edges.  This means that the outdegree of \(H\) is at least \(4\delta(G)-2\cdot 3\geq 2\lfloor n/2\rfloor +8-6=2\lfloor n/2\rfloor+2> 2(n/2-1)+2=n>n-\alpha_2^c(G)\), as desired.
\end{proof}

For girth \(6\) or more, it is difficult to improve on these results, since general \(\lambda_k\)-optimality results require a minimum degree of approximately \(n/4\).

 %Here's a fun condition to be \(\lambda_k\)-optimal:  if \(k\geq 3\), girth \(\gamma\geq 5\), minimum degree \(\delta\geq k\), and  diameter at most \(\gamma-4\) where \(\gamma\) is even and at most \(\gamma-3\) if \(\gamma\) is odd. (This is from \cite{optimal_based_on_girth}.)  So 

\subsection{Bipartite graphs with high valence}

Let \(G\) be a simple bipartite graph on \(n\) vertices with partite sets \(A\) and \(B\), with \(|A|=n_1\) and \(|B|=n_2\).  Since \(A\) and \(B\) are both independent sets, we have \(\alpha(G)\geq \max\{n_1,n_2\}\), implying that
\[\gon(G)\leq n-\max\{n_1,n_2\}=\min\{n_1,n_2\}.\]
 Sometimes we have equality; for instance, the gonality of the complete bipartite graph \(K_{n_1,n_2}\) is \(\min\{n_1,n_2\}\) by \cite[Example 4.3]{debruyn2014treewidth}.
A natural question to ask is then:  for which bipartite graphs $G$ does it hold that $\gon{(G)} = \min\{n_1,n_2\}$? Our first theorem in this subsection shows a sufficient assumption on \(\delta(G)\) to guarantee this equality.  In order to prove it, we present the following lemma.

\begin{lemma}\label{lemma:bipartite_independence}
Let \(G\) be a simple bipartite graph with partite sets \(A\) and \(B\), where \(|A|=n_1\) and \(|B|=n_2\) with \(n=n_1+n_2\).  Assume further that \(\delta(G)\geq n/4\). Then \(\alpha(G)=\max\{n_1,n_2\}\).
\end{lemma}

\begin{proof}
Certainly \(\alpha(G)\geq \max\{n_1,n_2\}\), since each partite set forms an independent set.  Suppose \(S\) is an independent set with vertices \(u\) and \(v\) in each of the two partite sets.  Then \(S\) can contain no neighbors of \(u\) or \(v\).  Since \(u\) and \(v\) are not neighbors and have no common neighbors, this gives us at least \(2\delta(G)\geq n/2\) vertices in \(S^C\).  This means that \(|S|=n-|S^C|\leq n-n/2= n/2\leq \max\{n_1,n_2\}\).  Thus no independent set can be larger than \(\max\{n_1,n_2\}\), and we have the desired result.
\end{proof}

\begin{theorem}\label{theorem:bipartite1}
Let \(G\) be a simple bipartite graph on \(n_1\) and \(n_2\) vertices, with \(n=n_1+n_2\geq 4\), such that \(\delta(G)\geq \frac{1}{2}(\lfloor n/2\rfloor +2)\).  Then \(\sn(G)=\gon(G)=\min\{n_1,n_2\}\).
\end{theorem}

\begin{proof} By Lemma \ref{lemma:bipartite_independence} and the fact that \(\frac{1}{2}(\lfloor n/2\rfloor +2)\geq n/4\), we have \(\alpha(G)=\max\{n_1,n_2\}\). 

In order to apply Theorem \ref{theorem:main} with \(\ell=2\), we need \(\lambda_2(G)\geq n- \alpha(G)=n-\max\{n_1,n_2\}=\min\{n_1,n_2\}\).  If \(\lambda_2(G)=\infty\), we are done; otherwise we have that \(G\) is \(\lambda_2\)-connected.  Since \(\delta(G)\geq \frac{1}{2}\left(\lfloor n/2\rfloor+2\right)\), we know that \(G\) is \(\lambda_2\)-optimal by Lemma \ref{lemma:hmm}.  Thus we have \(\lambda_2(G)=\xi_2(G)\geq 2\delta(G)-2\geq \lfloor n/2\rfloor +2-2=\lfloor n/2\rfloor \geq \min\{n_1,n_2\}\). This gives us \(\sn(G)=\gon(G)=n-\alpha(G)=\min\{n_1,n_2\}\), as desired.
\end{proof}

This gives us, for instance, the gonality of \emph{crown graphs}. The crown graph on \(2m\) vertices is a complete bipartite graph \(K_{m,m}\) with \(m\) pairwise non-incident edges deleted, so that the minimum valence is \(m-1\); for \(m\geq 4\), these graphs satisfy the hypotheses of our theorem, giving us a gonality of \(m\). To see that we cannot in general weaken our assumption on \(\delta(G)\) to \(\delta(G)\geq n/4\), consider the cycle graph \(C_8\) (a bipartite graph on \(4\) and \(4\) vertices), which has \(\delta(C_8)=2\geq |V(C_8)|/4\).  We have \(\gon(C_8)=2<4=\min\{n_1,n_2\}\), so the theorem would not hold.

We also obtain a result without making an assumption on minimum valence, as long as pairs of non-adjacent vertices have sufficiently high total valence.

\begin{theorem}\label{theorem:bipartite2}
Let \(G\) be a simple bipartite graph on \(n\geq 6\) vertices such that for any pair of non-adjacent vertices \(u\) and \(v\), we have \(\val(u)+\val(v)\geq 2\lfloor n/4\rfloor +3\).  Then \(\sn(G)=\gon(G)=n-\alpha_2^c(G)\).
\end{theorem}

\begin{proof}
Any simple bipartite graph has girth at least \(4\).  Thus in order to apply Theorem \ref{theorem:main} with \(\ell=4\), it suffices to show that \(\lambda_3(G)\geq n-\alpha_2^c(G)\). We know by Lemma \ref{lemma:hmm} and our assumptions on \(G\) that \(G\) is \(\lambda_3\)-optimal, so \(\lambda_3(G)=\xi_3(G)\).  A connected subgraph of a bipartite graph on three vertices must be a path. Letting \(u\) and \(v\) denote the non-adjacent vertices of such a path, we know that the outdegree of that path is at least \(\val(u)+\val(v)-2\geq 2\lfloor n/4\rfloor +3-2=2\lfloor n/4\rfloor +1.\)  Thus \(\lambda_3(G)\geq 2\lfloor n/4\rfloor +1\).

We know \(\alpha_2^c(G)\geq \alpha_1^c(G)\), and since \(G\) is bipartite we know \(\alpha_1^c(G)\geq \lceil n/2\rceil \) (as each partite set forms an independent set), meaning \(n-\alpha_2^c(G)\leq \lfloor \frac{n}{2}\rfloor\).  We thus have \(n-\alpha_2^c(G)\leq  \lfloor n/2\rfloor \leq 2\lfloor n/4\rfloor +1\leq \lambda_3(G)\), which allows us to conclude that \(\sn(G)=\gon(G)=n-\alpha_2^c(G)\), as desired.
\end{proof}

\subsection{Hypercube graphs}

The hypercube graph \(Q_n\), pictured for \(n=3\) in Figure \ref{fig:cube}, has gonality at most \(2^{n-1}\); this was first remarked in \cite[\S 4]{debruyn2014treewidth}, where it was conjectured to hold with equality for all \(n\).  There are several ways to see the upper bound holds; for instance, \(Q_n\) is a bipartite graph, and a positive rank divisor can be found by putting a chip on each vertex in one of the partite sets (each of which has \(2^{n-1}\) vertices). Alternatively, \(Q_n\) is a Cartesian product \(Q_{n-1}\square K_2\) and any graph of the form \(G\square K_2\) has a positive rank divisor of degree \(|V(G)|\) by \cite[Proposition 3]{aidun2019gonality}. One strategy to prove that \(\gon(Q_n)=2^{n-1}\) is to construct a scramble \(\cS\) on \(Q_n\) of order at least \(2^{n-1}\); it would then follow that \(2^{n-1}\leq ||\cS||\leq \sn(G)\leq \gon(G)\leq 2^{n-1}\), meaning that all numbers are equal to \(2^{n-1}\).  This strategy is implicit in \cite[Figure 4.2]{scramble} to give a concise proof that \(\gon(Q_3)=4\).
We will use this strategy to compute the gonality of \(Q_n\) for \(n=4\) and \(n=5\), and will show that this strategy does not work for \(n\geq 6\).  This both strengthens the cases known for a long-standing open conjecture, and demonstrates that scramble number will not be able to resolve it in general.

 \begin{figure}[hbt]
    \centering
  \includegraphics{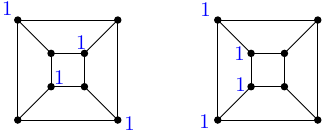}
    \caption{The cube graph \(Q_3\), with two positive rank divisors of degree \(4\).}
    \label{fig:cube}
    \end{figure}

The scrambles we will use are \(\mathcal{E}_k\) for some choice of \(k\), meaning that understanding \(\lambda_k(Q_n)\) will be important.  To this end, we recall the following lemma.  

\begin{lemma}[Theorem 3.1 in \cite{extra_edge_connectivity_hypercubes}] \label{lemma:lambdak}
Let \(n\geq 4\).  If \(\frac{2^{n-1}}{3}<k\leq 2^{n-1}\), then \(\lambda_k(Q_n)=2^{n-1}\).
\end{lemma}

\begin{lemma}\label{theorem:snQ4} The scramble \(\mathcal{E}_3\) on \(Q_4\) has order \(8\).
\end{lemma}

\begin{proof}
By Lemma \ref{lemma:lambdak}, $e(\mathcal{E}_3) \geq \lambda_3 (Q_4) = 2^{4-1} = 8$.  Since $||\mathcal{E}_3|| = \min\{h(\mathcal{E}_3), e(\mathcal{E}_3)\}$ it remains to show that \(h(\mathcal{E}_3)\geq 8\).  Let \(C\) be a hitting set for \(\mathcal{E}_3\), and consider the four disjoint copies of \(C_4\) in \(Q_4\) highlighted in Figure \ref{fig:cube4}.  The set \(C\) must contain at least two vertices from each \(C_4\); otherwise a connected subgraph on \(3\) vertices would be unhit.  Thus \(|C|\geq 2\cdot 4=8\), from which we conclude that $||\mathcal{E}_3|| = \min\{h(\mathcal{E}_3), e(\cS)\} \geq \min\{8,8\} = 8$, as desired. 
\end{proof}

\begin{proposition}\label{theorem:gonalityQ4} We have
$\sn(Q_4)=\gon(Q_4) = 8$.
\end{proposition}

\begin{proof}
From the previous lemma we have
\[8\leq \sn(Q_4)\leq \gon(Q_4)\leq 8,\]
implying our desired result.
\end{proof}

\begin{figure}[hbt]
    \centering
    \includegraphics{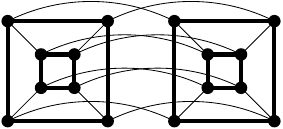}
    \caption{The hypercube $Q_4$ with four disjoint \(4\)-cycles highlighted.}
    \label{fig:cube4}
\end{figure}

Since \(Q_4=C_4\square C_4\), this result follows from \cite{scramble}, where it was proved that for the Cartesian product of two cycle graphs we have \(\sn(C_m\square C_n)=\gon(C_m\square C_n)=\min\{2m,2n\}\).  However, the next case of \(Q_5\) was open prior to the work in this paper.

\begin{lemma}
\label{theorem:snQ5} The scramble \(\mathcal{E}_6\) on \(Q_5\) has order \(16\).
\end{lemma}

\begin{proof}
By Lemma \ref{lemma:lambdak}, $e(\mathcal{E}_6) = \lambda_6 (Q_5) = 2^{5-1} = 16$.  Since $||\mathcal{E}_6|| = \min\{h(\mathcal{E}_6), e(\mathcal{E}_6)\}$ it remains to show that \(h(\mathcal{E}_6)\geq 16\).

Let \(H\subset V(Q_5)\) with \(|H|=15\), and suppose for the sake of contradiction that \(H\) is a hitting set of \(\mathcal{E}_6\).  Decompose the graph \(Q_5\) into four subgraphs \(A_1,A_2,B_1,B_2\), each isomorphic to \(Q_3\), with vertices in \(A_i\) connected to the corresponding vertices in \(A_j\) and \(B_i\). If one of these subgraphs contains two or fewer vertices from \(H\), then the other vertices form a connected subgraph on at least six vertices unhit by \(H\); thus each of \(A_1,A_2,B_1,B_2\) contains at least three vertices of \(H\).  Since there are \(15\) elements of \(H\), one of the four subgraphs has at most, and therefore exactly, three vertices from \(H\); without loss of generality assume it is \(A_1\).

If \(A_1-H\) is a connected subgraph on five vertices, then all \(10\) neighbors of \(A_1-H\) in \(A_2\cup B_1\) must be included in \(H\); otherwise we would have a connected subgraph on six vertices unhit by \(H\).  But this means that \(|V(B_2)\cap H|\leq 2<3\), a contradiction.  Thus \(A_1-H\) is not connected.  The only way to disconnect \(Q_3\) by removing $3$ vertices is to delete the $3$ neighbors of a single vertex. Thus \(A_1\) contains an isolated vertex \(v\) unhit by \(H\), and a connected subgraph \(K\) on four vertices unhit by \(H\), as pictured in Figure \ref{figure:Q_5_corner}.

\begin{figure}[hbt]
    \centering
    \includegraphics{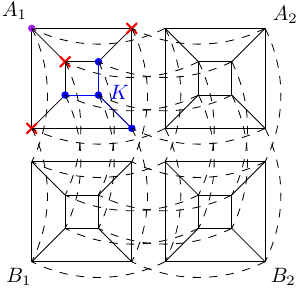}
    \caption{The placement of \(v\) and \(K\) inside \(A_1\); the three \(\times\)'s represent elements of \(H\)}
    \label{figure:Q_5_corner}
\end{figure}

Of the eight neighbors of \(K\) in \(A_2\cup B_1\), at least seven must be in \(H\) to avoid an unhit subgraph on six vertices.  Suppose for the moment that all eight are in \(H\).  We claim that \(A_2\) must have a fifth vertex in \(H\).  If not, then the four unhit vertices together with \(v\) form a connected subgraph on five vertices; and \(H\) must include the five neighbors of that subgraph in \(B_1\cup B_2\).  This gives a total of seven vertices from \(H\) in \(A_1\cup A_2\), and nine vertices from \(H\) in \(B_1\cup B_2\), a contradiction to \(|H|=15\); thus \(|V(A_2)\cap H|\geq 5\).  An identical argument shows \(|V(B_1)\cap H|\geq 5\), so \(|V(B_2)\cap H|\leq 2\).  It follows that \(B_2-H\) is a connected subgraph on at least six vertices unhit by \(H\), a contradiction.  Thus, of the eight neighbors of \(K\) in \(A_2\cup B_1\), exactly seven are in \(H\).

\begin{figure}[hbt]
    \centering
    \includegraphics{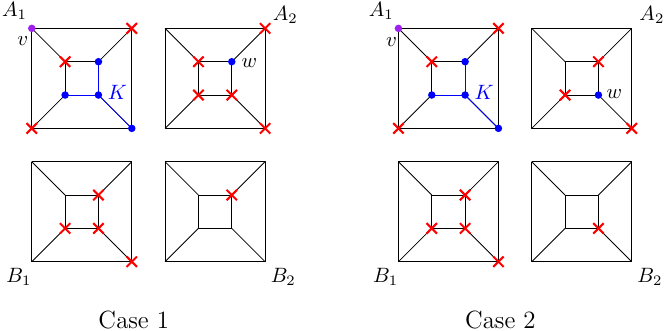}
    \caption{Two cases for the placement of \(H\) on \(Q_5\).}
    \label{figure:Q_5_cases_modified}
\end{figure}

Without loss of generality assume \(A_2\) contains the unique neighbor \(w\) of \(K\) unhit by \(H\). We split the proof into two cases:  where \(w\) is incident to one vertex in \(A_2\) neighboring \(L\), and where it is incident to three such vertices; these cases are illustrated in Figure \ref{figure:Q_5_cases_modified}.   In the first case, the two \(A_2\)-neighbors of \(w\) not neighboring \(K\), as well as the neighbor of \(w\) in \(B_2\), must also be in \(H\), as otherwise we would have an unhit graph on $6$ vertices; these vertices required in \(H\) are pictured on the left in Figure \ref{figure:Q_5_cases_modified}.  Note that this accounts for \(13\) elements of \(H\); the other two must be in \(B_2\), since \(|V(B_2)\cap H|\geq 3\).  But this leaves a connected subgraph on seven vertices unhit by \(H\), consisting of \(v\), two vertices from \(A_2\), and four vertices from \(B_1\), a contradiction.

Thus we must be in the second case, where all three of \(w\)'s neighbors in \(A_2\) also neighbor \(L\); this is pictured on the right in Figure \ref{figure:Q_5_cases_modified}, with the required elements of \(H\) marked.  So far this means that of vertices from \(H\) we have exactly three in \(A_1\), at least three in \(A_2\), at least four in \(B_1\), and at least one in \(B_2\).  We claim that there are at least four in \(A_2\), at least five in \(B_1\), and at least three in \(B_2\); this gives a total of \(15\), so in fact it must be exactly those amounts.  If there is no additional element of \(H\) in \(A_2\), then there is a $4$-vertex connected subgraph of \(A_2\) unhit by \(H\), which together with the isolated vertex in \(A_1\) forms such a subgraph on five vertices.  This necessitates another five elements of \(H\): one in \(B_1\), and four in \(B_2\), to prevent the subgraph from having six vertices; but this is too many, meaning that \(|A_2\cap H|\geq 4\).  A similar argument shows that \(|V(B_1)\cap H|\geq 4\).  Finally, \(|V(B_2)\cap H|\geq 3\) as already noted.  So, we have that \(|V(A_1)\cap H|=3\), \(|V(A_2)\cap H|=4\),  \(|V(B_1)\cap H|=5\), and \(|V(B_2)\cap H|=3\). 
    
Let \(L\) denote a subgraph of \(B_2\) with four vertices, namely the upper left vertex and its neighbors in \(B_2\).  Note that the eight neighbors of \(L\) in \(A_2\cup B_1\) are not yet known to be in \(H\).  We know that \(B_2-H\) must have a connected component on four or more vertices, since it is a cube with three vertices deleted.  We claim that that component contains at least two vertices, call them \(u_1\) and \(u_2\), of \(L\).  If the component has all five vertices of \(B_2-H\), then this certainly holds; and if the component has only four vertices of \(B_2-H\), then the configuration \(V(B_2)\cap H\) must be one of the ones pictured in Figure \ref{figure:placements_on_b2}, each of which yields two vertices of \(L\) in the connected component on four vertices.  To avoid introducing an unhit component of order six, of the four neighbors of \(u_1\) and \(u_2\) in \(A_2\cup B_1\) at least three must be in \(H\).  But that means either \(|A_2\cap H|\geq 5\) or \(|B_1\cap H|\geq 6\), a contradiction either to  \(|V(A_2\cap H)|= 4\) or to \(|V(B_1\cap H)|= 5\).

\begin{figure}[hbt]
    \centering
    \includegraphics{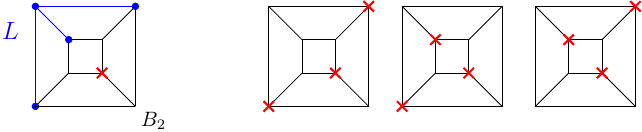}
    \caption{The subgraph \(L\) of \(B_2\); and all possible placements of \(H\) on \(B_2\) to disconnect it.  Each placement leaves two vertices of \(L\) in a connected component on \(4\) vertices.}
    \label{figure:placements_on_b2}
\end{figure}

In all cases we have reached a contradiction; therefore \(H\) must not be a hitting set.  We conclude that \(h(\mathcal{E}_6)\geq 16\), as desired.    
\end{proof}

\begin{theorem}\label{theorem:gonalityQ5} We have
$\sn(Q_5)=\gon(Q_5) = 16$.
\end{theorem}

\begin{proof}
From the previous lemma we have
\[16\leq \sn(Q_5)\leq \gon(Q_5)\leq 16,\]
implying our desired result.
\end{proof}

Sadly, we cannot use this strategy to prove that \(\gon(Q_n)=2^{n-1}
\) for \(n\geq 6\):  as the following theorem shows, the scramble number is too small for this.

\begin{theorem}  
If $n\geq 6$, then $\sn(Q_n)<2^{n-1}.$
\end{theorem}

\begin{proof}  

Let \(\cS\) be a scramble on \(Q_n\) of maximal order.  
We will prove that \(||\cS||<2^{n-1}\).

Recall that \(Q_n\) is a bipartite graph, partitioned into the vertices whose coordinates sum to an even number and the vertices whose coordinates sum to an odd number.  Let \(S\) be a set of any \(2^{n-1}-1\) vertices whose coordinates sum to an even number, with \(v\) the only such vertex left out from \(S\).  Then \(Q_n-S\) consists of isolated vertices (those vertices with odd coordinate sum that do not neighbor \(v\)) as well as a component \(C\) consisting of \(v\) and its \(n\) adjacent vertices.

If \(S\) is a hitting set for \(\cS\), then \(h(\cS)\leq 2^{n-1}-1\) and we are done.  Otherwise, there must be an egg \(E\) in \(V(Q_n)-S\). First suppose it is an isolated vertex.  Then either \((E,E^C)\) forms an egg-cut of size \(n\), giving \(e(\cS)\leq n<2^{n-1}\); or every egg in \(\cS\) intersects \(E\), in which case \(E\) is a hitting set of size \(1\) and \(h(\cS)=1<2^{n-1}\).  Now suppose the egg is not an isolated vertex; then it must be contained in \(C\).  If \((C,C^C)\) is not an egg-cut, then every egg intersects \(C\), so \(h(\cS)\leq |C|=n+1<2^{n-1}\).  Otherwise, \((C,C^C)\) is an egg-cut of size \(\textrm{outdeg}(C)=n(n-1)\).  In this case \(e(\cS)\leq n(n-1)<2^{n-1}\) since \(n\geq 6\).  In every scenario, we can conclude that \(||\cS||<2^{n-1}\).
\end{proof}

There are several other graph gonalities we can compute using our results for \(Q_4\) and \(Q_5\).  The \emph{\(n\)-dimensional folded cube graph \(FQ_n\)} is constructed by adding to \(Q_n\) edges connecting each vertex to the ``opposite'' vertex, at the maximal distance of \(n\). (Equivalently, one can construct \(FQ_n\) by gluing the opposite vertices of \(Q_{n+1}\) to one another.)  For instance, \(FQ_2\) is the complete graph on \(4\) vertices; \(FQ_3\) is the complete bipartite graph \(K_{4,4}\); and \(FQ_4\) and \(FQ_5\) are known as the Clebsch graph\footnote{This name was given in \cite{clebsch}, due to the graph's relation to the configuration of 16 lines on a quartic surface discovered by Alfred Clebsch.} and the Kummer graph\footnote{This name was given in \cite{kummer}, due to the graph's relation to the Kummer configuration of 16 points and 16 planes.}, respectively.

\begin{figure}[hbt]
    \centering
    \includegraphics[scale=0.6]{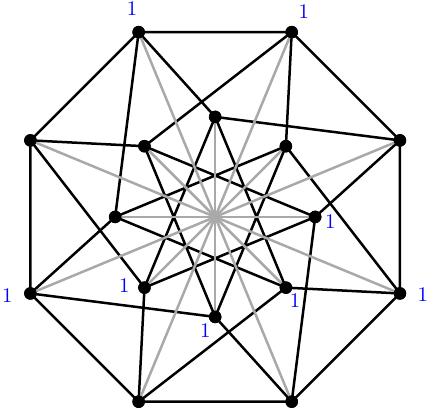}
    \caption{The Clebsch graph \(FQ_4\) with a divisor of positive rank.  The black edges form the hypercube \(Q_4\), with the grey edges connecting diagonals.}
    \label{figure:clebsch}
\end{figure}

The Clebsch graph \(FQ_4\) is illustrated in Figure \ref{figure:clebsch} with a divisor of degree \(8\).  This divisor has positive rank as can be checked manually, or by noting that the support of the divisor is a strong separator and applying Lemma \ref{lemma:strong_separator}.  Similary, the Kummer graph \(FQ_5\) has a positive rank divisor of degree \(16\); this can be seen most easily by noting that \(FQ_5\) is a bipartite\footnote{The folded cube graph \(FQ_n\) is bipartite if and only if \(n\) is odd by \cite[Theorem 2.1]{folded}.} graph with partite sets of size \(16\).  Thus we have \(\gon(FQ_4)\leq 8\) and \(\gon(FQ_5)\leq 16\).  In general, gonality is not monotone under taking subgraphs; however, scramble number is, as proved in \cite[Proposition 4.3]{scramble}. Thus by Proposition \ref{theorem:gonalityQ4} we have \(8=\sn(Q_4)\leq \sn(FQ_4)\leq \gon(FQ_4)=8 \), implying \(\sn(FQ_4)=\gon(FQ_4)=8\); and by Theorem \ref{theorem:gonalityQ5} we have \(16=\sn(Q_5)\leq \sn(FQ_5)\leq \gon(FQ_5)=16 \), implying \(\sn(FQ_5)=\gon(FQ_5)=16\).

\section{The complexity of egg-cut problems}
\label{section:np_hard}

It was shown in \cite{echavarria2021scramble} that computing \(\sn(G)\) is NP-hard.  It was  remarked in that paper that computing the order of a single scramble \(\mathcal{S}\) is also NP-hard; this is because \(h(\mathcal{S})\) is NP-hard to compute, even when all subsets overlap and we have \(e(\mathcal{S})=\infty\), yielding \(||\mathcal{S}||=h(\mathcal{S})\).  In this section we consider the complexity of computing \(e(\mathcal{S})\).

We introduce the following two problems related to egg-cuts.

\noindent \textbf{Egg-Cut Finiteness (ECF):}
\begin{enumerate}
    \item \textbf{Instance:} A scramble $\mathcal{S} = \{V_1,\ldots, V_s\}$ on a graph $G$.
    \item \textbf{Question:} Is $e(S)$ finite?
\end{enumerate}

%\noindent \textbf{Egg-Cut Bound (ECB):}
%\begin{enumerate}
 %   \item \textbf{Instance:} A scramble $\mathcal{S} = \{E_1, ..., E_s\}$ on a graph $G$, and an integer \(\ell\).
%    \item \textbf{Question:} Do we have \(e(\mathcal{S})\leq \ell\)?
%\end{enumerate}

\noindent \textbf{Egg-Cut Number (ECN):}
\begin{enumerate}
    \item \textbf{Instance:} A scramble $\mathcal{S} = \{V_1, \ldots, V_s\}$ on a graph $G$.
    \item \textbf{Output:} \(e(\mathcal{S})\).
\end{enumerate}

In studying the complexity of these problems, there are two natural ways to measure the size of the input:  in terms of the graph \(G\) alone (i.e.~in terms of \(|V(G)|\) and \(|E(G)|\)), and in terms of \(G\) as well as \(s=|\mathcal{S}|\).  We refer to the first as the \(G\)-interpretation, and the second as the \((G,s)\)-interpretation.  We remark that under either interpretation ECN is at least as hard as ECF, since computing \(e(\mathcal{S})\) immediately answers whether it is finite.

To study these problems, we recall the following problem from \cite{montejano2016complexity}:

\noindent \textbf{Existential Restricted Edge-Connectivity (EREC):}
\begin{enumerate}
    \item \textbf{Instance:} A graph \(G\) and a positive integer \(k\).
    \item \textbf{Question:} Is \(G\) \(\lambda_k\)-connected?
\end{enumerate}

\begin{theorem}
Under the \(G\)-interpretation, the problems ECF and ECN are both NP-hard, even for graphs of maximum degree \(5\).  Moreover, ECF is NP-complete.
\end{theorem}

\begin{proof}
Given an instance \((G,k)\) of EREC, consider the instance \((G,\mathcal{E}_k)\) of ECF.  The answer to EREC on \((G,k)\) is yes if and only if \(G\) is \(\lambda_k\)-connected, if and only if \(e(\mathcal{E}_k)<\infty\), if and only if the answer to ECF on \((G,\mathcal{E}_k)\) is yes. Since EREC is NP-complete even for graphs of maximum degree \(5\) \cite[Theorem 7]{montejano2016complexity}, we know that ECF is NP-hard for such graphs; and as ECN is at least as hard as ECF, the same holds for it as well.

To see that ECF is in NP, we need to show that there exists a polynomial-time verifiable certificate for any yes instance \((G,\mathcal{S})\).  Note that \((G,\mathcal{S})\) is a yes instance if and only if \(e(\mathcal{S})<\infty\), if and only if there exist two disjoint eggs \(V_i,V_j\in\mathcal{S}\).  Such a pair of eggs is a certificate of a yes instance, and the disjointness of \(V_i\) and \(V_j\) can be verified in \(O(|V(G)|^2)\) time.  Since ECF is both NP-hard and in NP, we conclude that ECF is NP-complete.
\end{proof}

We now prove that under the \((G,s)\)-interpretation, both of our egg-cut problems are easy.

\begin{theorem}
Under the \((G,s)\)-interpretation, there is a polynomial-time algorithm to solve ECF and ECN.
\end{theorem}

\begin{proof}
Suppose we are given an instance \((G,\mathcal{S})\) of ECN.  There are \(s \choose 2\) pairs \(V_i,V_j\) of eggs.  Since  \({s \choose 2}=\frac{s(s-1)}{2}\) is polynomial in \(s\), it suffices to show that we can compute in polynomial time the size of a minimum egg-cut separating \(V_i\) from \(V_j\).  Without loss of generality we assume \(V_i\cap V_j=\emptyset\) (this can be checked in \(O(n^2)\) time).

Define the graph \(G_{i,j}\) as follows.  Collapse the subgraph \(G[V_i]\) to a vertex \(u\), possibly introducing multi-edges if multiple vertices of \(V_i\) are incident to the same vertex of \(V_i^C\).  Do the same for \(G[V_j]\) to a vertex \(v\).  Given a subset \(A\subset V(G)\) and its complement \(B=A^C\subset V(G)\), we can consider the corresponding subsets \(A'\) and \(B'\) in \(G_{i,j}\).  By construction, we have \(|V_G(A,B)|=|V_{G_{i,j}}(A',B')|\); it follows that the smallest egg-cut separating \(V_i\) from \(V_j\) has the same size as the smallest \(uv\)-edge-cut in \(G_{i,j}\).  This is the undirected version of a max-flow/min-cut problem, which can be solved in \(O(n|E(G)|^2)\) time using the Edmonds-Karp algorithm \cite{edmonds_karp}.  We conclude that \(e(\mathcal{S})\) can be computed in \(\textrm{poly}(n,s)\)-time.

Our polynomial-time algorithm for ECN also answers ECF in the same time, completing the proof.
\end{proof}

\noindent \textbf{Acknowledgements.}  The authors were supported by Williams College and the SMALL REU, and by the NSF via grants DMS-1659037 and DMS-2011743.

\bibliographystyle{abbrv}

\end{document}